\documentclass[a4paper,11pt]{article}
\usepackage{latexsym}
\usepackage{amssymb}
\usepackage{enumerate}
\usepackage{xcolor}
\usepackage{commath}
\usepackage{array}
\usepackage{amsfonts}
\usepackage{amsmath,amsthm}
\usepackage{hyperref}
\usepackage{algorithm}
\usepackage{algorithmic}
\usepackage{framed}
\usepackage{boites}
\usepackage{graphicx}
\usepackage{cases}
\usepackage{mathtools}

\newenvironment{@abssec}[1]{%
\if@twocolumn

\section*{#1}%
\else

\vspace{.05in}\footnotesize
\parindent .2in
{\upshape\bfseries #1. }\ignorespaces
\fi}

{\if@twocolumn\else\par\vspace{.1in}\fi}

\newenvironment{keywords}{\begin{@abssec}{\keywordsname}}{\end{@abssec}}

\newenvironment{AMS}{\begin{@abssec}{\AMSname}}{\end{@abssec}}

\newcommand\keywordsname{Key words}
\newcommand\AMSname{AMS subject classifications}
\newcommand\AMname{AMS subject classification}
\newcommand\restr[2]{{
\left.\kern-\nulldelimiterspace 
#1 
\vphantom{|} 
\right|_{#2} 
}}
\newtheorem{theorem}{Theorem}[section]
\newtheorem{lemma}[theorem]{Lemma}
\newtheorem{corollary}[theorem]{Corollary}
\newtheorem{proposition}[theorem]{Proposition}
\newtheorem{remark}[theorem]{Remark}
\newtheorem{definition}[theorem]{Definition}
\newtheorem{problem}{Problem}

\newtheorem{mainthm}{Theorem}

\newtheorem{thm}{Theorem}

\newtheorem{lem}[thm]{Lemma}

\newcommand{\NN}{\mathbb{N}}

\newcommand{\RR}{\mathbb{R}}

\def\XXint#1#2#3{{\setbox0=\hbox{$#1{#2#3}{\int}$}
\vcenter{\hbox{$#2#3$}}\kern-.5\wd0}}

\newcommand{\link}{\mathop{\circ\kern-.35em -}}

\newcommand{\ol}{\overline}
\newcommand{\pa}{\partial}

\newcommand{\dv}{\mathop{\mathrm{div}}}

\newcommand{\gr}{\nabla}

\newcommand{\al}{\alpha}
\newcommand{\be}{\beta}
\newcommand{\ga}{\gamma}
\newcommand{\Ga}{\Gamma}

\newcommand{\De}{\Delta}

\newcommand{\la}{\lambda}
\newcommand{\La}{\Lambda}

\newcommand{\te}{\theta}

\newcommand{\om}{\omega}
\newcommand{\Om}{\Omega}
\newcommand{\rn}{{\mathbb{R}}^N}
\newcommand{\sg}{\sigma}

\newcommand\setbld[2]{\left\{ #1 \; :\; #2\right\}}

\newcommand{\tin}{{\text{in }}}
\newcommand{\ton}{{\text{on }}}

\newcommand{\tfor}{{\text{for }}}
\newcommand{\id}{{\rm Id}}
\newcommand\pp[1]{\left( #1\right)}

\newcommand\qq[1]{\left[ #1\right]}
\newcommand{\inv}{^{-1}}
\newcommand{\cdottone}{{\boldsymbol{\cdot}}}

\newcommand{\dato}{\restr{\frac{d}{dt}}{t=0}}

\newcommand{\C}{\mathcal{C}}
\newcommand{\cC}{\mathcal{C}}

\newcommand{\cF}{\mathcal{F}}
\newcommand{\cG}{{\mathcal G}}
\newcommand{\cH}{{\mathcal H}}

\newcommand{\cJ}{{\mathcal J}}

\newcommand{\cO}{{\mathcal O}}

\newcommand{\cX}{\mathcal{X}}
\newcommand{\cY}{\mathcal{Y}}

\title{How to construct parametrized families of free boundaries near nondegenerate solutions
\thanks{This research was partially supported by the
Grant-in-Aid for Research Activity Start-up (No. 20K22298) of the Japan Society for the Promotion of Science.}
}

\author{Lorenzo Cavallina 
}
\date{}

\begin{document}

\maketitle

\begin{abstract}
In this paper, we introduce the notion of \emph{variational} free boundary problem. Namely, we say that a free boundary problem is variational if its solutions can be characterized as the critical points of some shape functional. Moreover, we extend the notion of nondegeneracy of a critical point to this setting.
As a result, we provide a unified functional-analytical framework that allows us to construct families of solutions to variational free boundary problems whenever the shape functional is nondegenerate at some given solution.

As a clarifying example, we apply this machinery to construct families of nontrivial solutions to the two-phase Serrin's overdetermined problem in both the degenerate and nondegenerate case.
\end{abstract}

\begin{keywords}
overdetermined problem, free boundary problem, nondegenerate critical points, shape derivatives, implicit function theorem, two-phase, Serrin's overdetermined problem.
\end{keywords}

\begin{AMS}
35N25, 35J15, 35Q93, 35B32
\end{AMS}

\pagestyle{plain}
\thispagestyle{plain}

\section{Introduction}
\subsection{The general framework}
In this paper, we consider a class of free-boundary problems that we will call \emph{variational}. We will refer to a free boundary problem as \emph{variational} if its solutions are characterized as being the critical shapes of some shape functional. Among many, two famous examples are: for some constant $c>0$ find a bounded domain $\Om\subset\rn$ such that
\begin{enumerate}[(1)]
\item 
its mean curvature $H$ is constantly equal to $c$ on $\pa\Om$ \cite{Ale1958}.
\item 
the solution $u_\Om$ of the boundary value problem
\begin{equation*}
-\De u_\Om =1 \quad \tin\Om, \quad u_\Om=0 \quad \ton \pa\Om
\end{equation*}
satisfies $|\gr u_\Om|\equiv c$ on $\pa\Om$ \cite{Se1971}.
\end{enumerate}
It is easy to check that the solutions of the two overdetermined problems $(1)$ and $(2)$ coincide with the critical shapes of the functionals
\begin{equation*}
J_1(\Om)=|\pa\Om|-c|\Om| \quad \text{and} \quad J_2(\Om)=\int_\Om |\gr u_\Om|^2 - c^2|\Om| \end{equation*}
respectively (thus, in particular, both $(1)$ and $(2)$ are \emph{variational} free boundary problems).
Well-known results \cite{Ale1958, Se1971} state that, under suitable regularity assumptions, the solutions to both problems turn out to be balls.
Of course, one can also consider free boundary problems where only a portion of the boundary is free (to be determined) or ones that depend on multiple parameters. This usually gives rise to a whole family of nontrivial (non-radially symmetric) solutions. The idea of overdetermined problems that admit a family of nontrivial solutions is not new. Indeed, such problems have been studied since a long time ago. Although it would be impossible to give an exhaustive list of the known results in the field, we refer the interested reader to \cite{beurling, flucher rumpf, altcaff, acker, HS97, BHS2014, DvEPs,CY1, CYisaac, kamburov sciaraffia, henrot onodera, gilsbach onodera, gilsbach stollenwerk} and the references therein.

In this paper, we propose a systematic way to study the local behavior of the (parametrized) families of solutions of variational overdetermined problems.

Let $\Om\subset\rn$ ($N\ge2$) be a sufficiently smooth open set with compact boundary $\pa \Om$ and let $U$ be a bounded open neighborhood of $\pa\Om$ with sufficiently smooth boundary. Moreover, let $\Theta$ and $\Theta_{\rm reg}$ be two Banach spaces of $\rn$ valued functions that satisfy:
\begin{equation*}
C_0^\infty(\ol U, \rn)\subset \Theta_{\rm reg}\subset C_0^1(\ol U,\rn)\subset \Theta\subset W_0^{1,\infty}(U,\rn).
\end{equation*}
Throughout this paper, when there is no chance of confusion, the same symbol will be used to denote both a function defined on some subset of $\rn$ and its extension by zero to the whole space.
Now, for $\theta\in \Theta$ let $\Om_\theta:=(\id+\theta)(\Om)$, where $\id:\rn\to\rn$ is the identity mapping and set \begin{equation*}
\cO:=\setbld{\Om_\theta}{\theta\in\Theta, \quad \norm{\theta}_{W^{1,\infty}(U,\rn)}<1}.
\end{equation*}

Let now $m\in\NN$, $\La$ be a Banach space of ``parameters" and assume that for all $\om\in\cO$ there exists a mapping $g_\om:\La\to L^2(\pa\om,\RR^m)$.
A general parametrized free boundary problem in a neighborhood of a solution $\Om$ can then be formalized as follows:
\begin{problem}\label{pb 1}
For given $\la\in\La$, find $\om\in\cO$ such that
\begin{equation*}
g_\om(\la)=0 \quad \ton \pa\om.
\end{equation*}
\end{problem}
In what follows, we will assume that for all $\om\in\cO$ there exists an open neighborhood $\Theta'$ of $0\in\Theta$ such that the set $\om_\te:=(\id+\te)(\om)$ is well defined. Moreover, we will also assume that Problem \ref{pb 1} is \emph{variational}, that is, there exists a parametrized shape functional $J:\cO\times \La\to \RR^m$ such that the map
\begin{equation}\label{J_om def}
\begin{aligned}
\cJ_\om:\quad & \Theta'\times \La\to \RR^m, \\
\ & (\theta,\la)\mapsto J(\om_\theta,\la)
\end{aligned}
\end{equation}
is Fr\'echet differentiable and its partial Fr\'echet derivative with respect to the first variable is given by
\begin{equation}\label{first derivative}
\pa_\theta \cJ_\om(0,\la)[\theta_0]=\int_{\pa\om}g_\om(\la)\ \theta_0\cdot n_\om \quad \text{for all }\theta_0 \in\Theta,
\end{equation}
where $n_\om$ denotes the outward unit normal of $\pa\om$ (later, in Remark \ref{first derivatives satisfy the structure hypothesis}, we explain how restrictive the structure formula \eqref{first derivative} actually is).

Moreover, let $X$ denote a Banach space of real valued functions on $\pa\Om$ and assume that there exists a Fr\'echet differentiable ``extension operator" $E:X\times \La\to \Theta_{\rm reg}$ such that $E(\cdottone,0):X\to\Theta_{\rm reg}$ is a bounded linear operator that satisfies
\begin{equation}\label{restr E}
\restr{\left(E( \xi,0)\right)}{\pa\Om}= \xi n.
\end{equation}
Now, let $Y\subset L^2(\pa\Om,\RR^m)$ be a Banach space and suppose that for some open neighborhood $\Theta'_{\rm reg}$ of $0\in \Theta_{\rm reg}$, the mapping
$
(\theta,\la)\mapsto g_{\Om_\theta}(\la)\circ\pp{\id+\theta}\in Y
$
is Fr\'echet differentiable in a neighborhood of $(0,0)\in \Theta'_{\rm reg}\times \La'$. For $(\xi,\la)$ small, let $\Om_{\xi,\la}$ denote the set $\Om_{E_\Om(\xi,\la)}$.
By composition, there exists a neighborhood $X'$ of $0\in X$ such that the mapping
\begin{equation}\label{j}
\begin{aligned}
j:\quad & X'\times \La' \to \RR^m,\\
\ & (\xi,\la) \longmapsto J(\Om_{\xi,\la},\la)
\end{aligned}
\end{equation}
is twice Fr\'echet differentiable. Moreover, there exists a bounded linear operator $Q: X \to Y$ such that
\begin{equation}\label{Q}
\pa_{xx}^2\ j(0,0)[\xi,\eta]= \int_{\pa\Om} Q(\xi) \eta \quad \text{for all }\xi,\eta\in X.
\end{equation}
(The proof of existence and the actual construction of $Q$ are dealt with in section $2$.)
Now, employing the structure formulas \eqref{first derivative} and \eqref{Q}, we say that $\Om$ is a \emph{nondegenerate} critical shape for $J$ at $\la=0$ if the following two conditions hold:
\begin{enumerate}[(i)]
\item $g_\Om(0)=0$ on $\pa\Om$ (\emph{criticality});
\item the mapping $Q$ is a bijection between $X$ and Y (\emph{nondegeneracy}).
\end{enumerate}

After a long preparation, we are now ready to state a result that links the nondegeneracy of a critical shape $\Om$ with the existence of a parametrized family of solutions to Problem \ref{pb 1} in a neighborhood of $\Om$.
\begin{mainthm}\label{thm I}
Let the notation be as above. Suppose that $\Om$ is a nondegenerate critical shape for the shape functional $J$ at $\la=0$. Then, there exists open neighborhoods $X''$, $\La''$ of $0\in X$ and $0\in \La$ respectively and a continuous map $\La''\ni \la \mapsto \xi(\la)\in X''$ such that the set $\Om_{\xi(\la),\la}$ is a solution to Problem \ref{pb 1}. Moreover, for $(\xi,\la)\in X''\times \La''$, the set $\Om_{\xi,\la}$ is a solution to Problem \ref{pb 1} if and only if $\xi=\xi(\la)$.
\end{mainthm}
\begin{remark}
The definition of \emph{nondegeneracy} of a critical point used in Theorem \ref{thm I} can be thought of as a generalization of that used by Smale, Palais and Tromba in \cite{smale, Palais 69, tromba}. Indeed, by considering two (possibly distinct) Banach spaces $X$ and $Y$, we can take into account the ``derivative loss" that usually occurs when dealing with shape derivatives.
\end{remark}
\begin{remark}
If equation
\begin{equation*}
Q(\xi)=\eta\qquad \text{for } (\xi,\eta)\in X\times Y
\end{equation*}
satisfies the Fredholm alternative (as it is often the case if one appropriately chooses the spaces $X$ and $Y$), then the nondegeneracy assumption $(ii)$ can be simply rewritten as $\ker Q=\{0\}$.
\end{remark}
\begin{remark}
We purposely chose a very general setting concerning the action of the parameters $\la\in\La$. As a matter of fact, in some applications it might be useful to consider the case where the functional $J$ depends on the parameter $\la$ through some auxiliary shape (for instance, the parameter $\la$ could be used to ``encode" the perturbation of some set, as will be done for Problem \ref{pb 2} in the following subsection). Similarly, we chose to consider the general case where the extension operator $E$ also depends on the parameter $\la$.
\end{remark}

\subsection{Applications to the two-phase Serrin's overdetermined problem}
In what follows, we will consider an example problem where the machinery of Theorem~\ref{thm I} can be applied to construct parametrized families of nontrivial solutions.

Given a bounded open set $D\subset\rn$ and a positive constant $\sg_c\ne 1$, let $\sg$ denote the following piece-wise constant function
\begin{equation}\label{sigma}
\sg=\sg_c\ \cX_D+\cX_{\rn\setminus D},
\end{equation}
where $\cX_A$ is the characteristic function of the set $A$ (i.e., $\cX_A(x)=1$ if $x\in A$ and $\cX_A(x)=0$ otherwise) and consider the following overdetermined problem.
\begin{problem}\label{pb 2}
Let $D\subset\rn$ be a bounded open set. Find a domain $\Om\supset \ol D$ such that the solution $u$ to the boundary value problem
\begin{equation}\label{2ph eq}
\begin{cases}
-\dv\left(\sg\gr u\right)=1 \quad \text{in }\Om, \\
u=0\quad \text{on }\pa\Om,\\
\end{cases}
\end{equation}
also solves the overdetermined condition
\begin{equation} \label{2ph odc}
|\gr u| \equiv c \quad \text{on }\pa\Om
\end{equation}
for some positive constant $c\in\RR$.
\end{problem}
Let $(D,\Om)$ be a pair of bounded open sets satisfying $\ol D\subset \Om$. Moreover, let $\pa\Om$ be at least of class $C^{3,\al}$ and let the pair $(D,\Om)$ be a solution of Problem \ref{pb 2}. Set
\begin{equation}\label{some definitions}
\begin{aligned}
X:=C^{2,\al}(\pa\Om), \quad Y:=C^{1,\al}(\pa\Om), \quad \Phi:=W^{1,\infty}_0(V,\rn), \\
\Theta_{\rm reg}:=\setbld{\te\in C^{2,\al}(\rn,\rn)}{\restr{\te}{\rn\setminus U}\equiv 0}.
\end{aligned}
\end{equation}
Here $V$ is an open set such that $\ol D\subset V\subset \ol V \subset \Om$, while $U$ is a bounded open neighborhood of $\pa\Om$ with sufficiently smooth boundary. Also, let $E: X\to \Theta_{\rm reg}$ be a bounded linear extension operator such that,
\begin{equation}\label{extension E}
\restr{E(\xi)}{\pa\Om}= \xi n,\quad \text{for all }\xi\in X,
\end{equation}
where $n$ denotes the outward unit normal to $\pa\Om$.
Let
\begin{equation}\label{Om_xi D_phi}
\Om_\xi := (\id + E(\xi))(\Om), \quad D_\varphi := (\id + \varphi)(D).
\end{equation}
Notice that, for $(\xi,\varphi)\in X\times\Phi$ sufficiently small, the sets $\Om_\xi$ and $D_\varphi$ are well defined and the inclusion $\ol{D_\varphi}\subset \Om_\xi$ holds.
Let now
\begin{equation*}
\La:= \Phi\times C^{1,\al}(\Om\cup \ol U)\times \RR.
\end{equation*}
\begin{definition}[Solution of Problem \ref{pb 2} with respect to $\la\in\La$]
For any $\la=(\varphi,f, s)\in\La$ and $\xi\in X$ small enough, we say that $(D_\varphi,\Om_\xi)$ is a solution of Problem \ref{pb 1} with respect to the parameters $\la=(\varphi,f, s)$ if the solution of the boundary value problem
\begin{equation}\label{2ph eq pert}
\begin{cases}
-\dv\left( \sg\gr u\right)=1 \quad \text{in }\Om_\xi, \\
u=0\quad \text{on }\pa\Om_\xi,\\
\end{cases}
\end{equation}
where
\begin{equation}\label{sigma pert}
\sg= (\sg_c+s)\cX_{D_\varphi} + \cX_{\rn\setminus D_\varphi}
\end{equation}
also satisfies the overdetermined condition
\begin{equation*}
|\gr u(x)| = c+f(x) \quad \text{for }x\in\pa\Om_{\xi}.
\end{equation*}
\end{definition}
\begin{figure}[h]
\centering
\includegraphics[width=\linewidth]{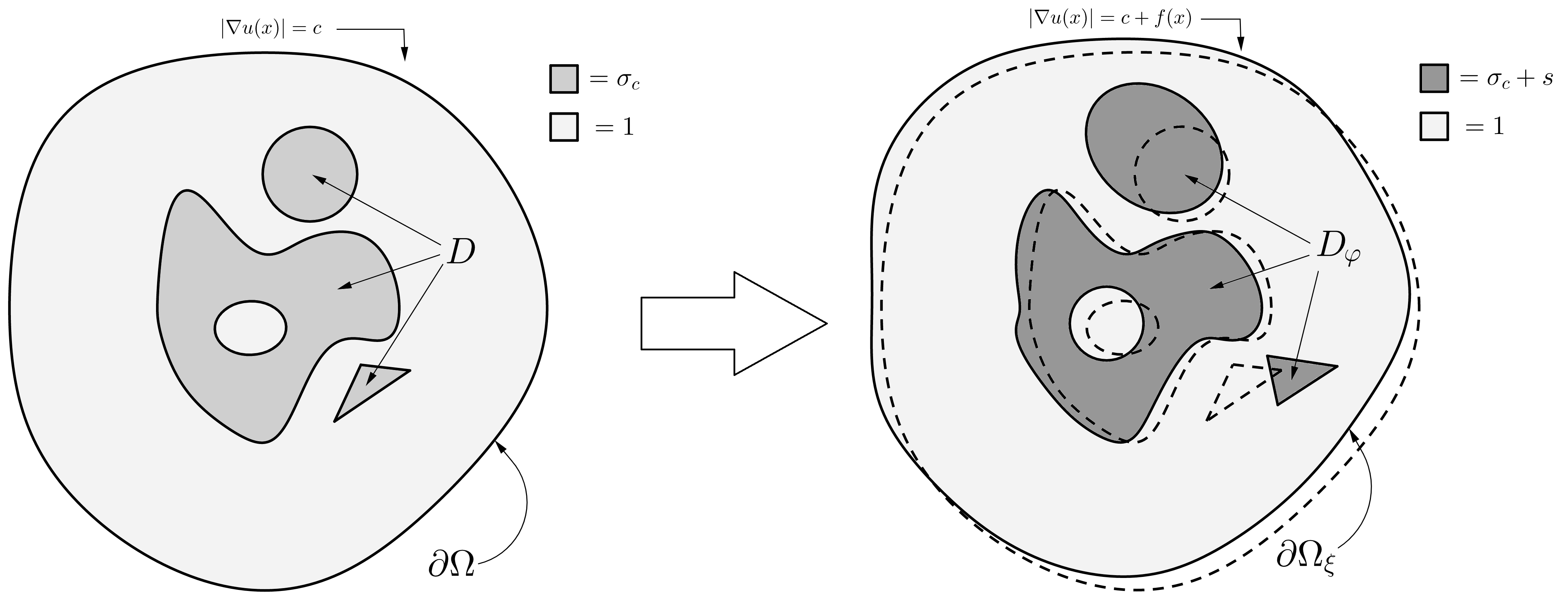}
\caption{Solution of the perturbed problem with respect to $\la=(\varphi,f,s)$}.
\label{perturbed problem}
\end{figure}

By applying Theorem \ref{thm I} to Problem \ref{pb 2} we obtain the following result.
\begin{mainthm} \label{thm II}
Let $(D,\Om)$ be a nondegenerate solution of Problem \ref{pb 2} (in the sense of Definition \ref{def non deg}) and let the notation be as above.
Then there exist neighborhoods $\La'$ of $0\in\La$ and $X'$ of $0\in X$, and a map $\widetilde\xi\in C^\infty(\La',X')$ such that the following hold:
\begin{enumerate}[(1)]
\item For all $\la=(\varphi, f, s)\in\La'$, the pair $(D_\varphi, \Om_{\widetilde\xi(\la)})$ is a solution of
Problem \eqref{pb 1} with respect to the parameters $\la=(\varphi, f, s)$.
\item If $(D_\varphi, \Om_\xi)$ is a solution of Problem \ref{pb 1} with respect to the parameters $\la=(\varphi, f, s)$ for some $(\xi,\la)\in X'\times \La'$, then $\xi=\widetilde \xi(\la)$.
\end{enumerate}
\end{mainthm}

When $\Om$ is a ball and $\sg_c=1$ (one-phase case) $\Om$ is a \emph{degenerate} solution of Problem \ref{pb 2} and thus Theorem \ref{thm II} does not hold. Nevertheless, by restricting the perturbation space, one can still manage to apply the machinery of Theorem \ref{thm I}.

Before stating our result, let us first introduce the notation.
Let $\Om$ be a ball and $D$ be an open set with $\ol D\subset\Om$ and $\sg_c=1$ (that is, the pair $(D,\Om)$ is a solution to Problem \ref{pb 2} for some $c$ depending on the radius of $\Om$). Moreover, let
\begin{equation}\label{notaton 1}
\La:=\Phi\times C^{1,\al}(\Om\cup \ol U)\times \RR\times \cY_1(\pa\Om), \quad
X:= C_{\rm bar}^{2,\al}(\pa\Om),\quad Y:=C_{\rm bar}^{1,\al}(\pa\Om),
\end{equation}
where
\begin{equation}\label{notation 2}
\begin{aligned}
\cY_1(\pa\Om)&:=\setbld{p:\pa\Om\to\RR}{p(x)=a\cdot x\quad \ton \pa\Om, \ \text{for some }a\in\rn},\\
C_{\rm bar}^{k,\al}(\pa\Om)&:=\setbld{\xi\in C^{k,\al}(\pa\Om)}{\langle \xi,p\rangle_{L^2(\pa\Om)}=0 \quad \text{for all }p\in\cY_1}, \quad k=1,2.
\end{aligned}
\end{equation}

\begin{mainthm}\label{thm III}
Let the notation be as above. There exist neighborhoods $\La'$ of $0\in\La$ and $X'$ of $0\in X$, and a map $\widetilde \xi\in C^\infty(\La',X')$ such that the following hold:
\begin{enumerate}[(1)]
\item For all $\la=(\varphi, f, s,\eta)\in\La'$, the pair $\left(D_\varphi, \Om_{\widetilde\xi(\la)+\eta}\right)$ is a solution to the overdetermined problem given by \eqref{2ph eq pert}, \eqref{sigma pert} and overdetermined condition
\begin{equation}\label{odc bar}
\big(|\gr u|^2-(c+f)^2\big)\circ \pp{\id+(\xi+\eta)n)} \in\cY_1(\pa\Om).
\end{equation}
\item If $(D_\varphi, \Om_{\xi+\eta})$ is a solution to the overdetermined problem given by \eqref{2ph eq pert}, \eqref{sigma pert} and overdetermined condition \eqref{odc bar} for some $(\xi,\la)\in X'\times \La'$, then $\xi=\widetilde \xi(\la)$.
\end{enumerate}
\end{mainthm}

This paper is organized as follows. In section 2 we give a proof of Theorem \ref{thm I} utilizing the implicit function theorem for Banach spaces (Theorem \ref{ift} of page \pageref{ift}). In the remaining sections, we show how to apply Theorem \ref{thm I} to Problem \ref{pb 2}. The nondegenerate case is dealt with in section 3 (where we prove Theorem \ref{thm II}), while the degenerate case is dealt with in section 4 (where we prove Theorem \ref{thm III}). Finally, the Appendix covers the technical details of the construction of bounded linear extension operators for $C^{k,\al}$ functions.

\section{Proof of Theorem \ref{thm I}}

\subsection{Preliminaries: the structure theorem for shape derivatives}
\begin{thm}[Structure theorem, \cite{structure}]\label{struct thm}
Let $J:\cO\to \RR^m$ be a shape functional. Consider a fixed domain ${\om}\in\cO$, a smooth open neighborhood $U$ of $\pa\om$ and define the map
\begin{equation*}
\begin{aligned}
\cJ:\quad & \Theta'\to \RR^m, \\
\ & \theta\mapsto J\left( \om_\te\right),
\end{aligned}
\end{equation*}
where $\Theta'$ is a sufficiently small neighborhood of $0\in C^1(\ol U, \rn)$ and $\om_\te:=(\id+\theta)(\om)$. Moreover, let $n$ denote the outward unit normal vector to $\pa\om$ and assume that the functional $\cJ$ is differentiable at $0\in\Theta'$. Then there exists a continuous linear map $\ell_\om:C^1(\pa\om)\to \RR^m$ such that
\begin{equation}
\quad \cJ'(0)[\te]=\ell_\om(\te\cdot n), \quad \text{for all }\te\in C^1(\ol U,\rn).
\end{equation}
\end{thm}
\begin{remark}\label{first derivatives satisfy the structure hypothesis}
In what follows we will explain why the structural hypothesis \eqref{first derivative} is not as restrictive as it might seem at first sight. Let us assume $C^1(\ol U,\rn)\subset\Theta$. First of all, if the mapping $\cJ_\om:\Theta\to \RR^m$ defined by \eqref{J_om def} is Fr\'echet differentiable at $0\in\Theta$, then, in particular, by Theorem \ref{struct thm}:
\begin{equation*}
\cJ'_\om(0)[\theta]=\ell_\om (\theta\cdot n) \quad \text{for all }\theta\in C^1(\ol U,\rn).
\end{equation*}
Now, for $1\le k\le m$, let $\pi_k:\RR^m\to \RR$ denote the projection onto the $k$\textsuperscript{th} coordinate. Then, $\pi_k\circ \ell_\om:C^1(\pa\om)\to \RR$ is a bounded linear functional. Now, by the Hahn--Banach theorem, $\pi_k\circ \ell_\om$ admits a continuous linear extension $\ga_k:L^2(\pa\om)\to\RR$. In particular, by the Riesz representation theorem there exist functions $g_k:L^2(\pa\om)$ such that
\begin{equation*}
\ga_k(\cdottone)= \langle g_k, \cdottone\ \rangle_{L^2(\pa\om)}\qquad (k=1,2,\dots,m).
\end{equation*}
This implies that, for all $\te\in C^1(\ol U,\rn)$, the Fr\'echet derivative of $\cJ_\om$ at $0$ can be written as
\begin{equation}\label{J' on C^1}
\cJ'_\om(0)[\te]= \ell_\om(\te\cdot n)=
\begin{pmatrix}
\pi_1\circ\ell_\om (\te\cdot n)\\
\vdots \\
\pi_m\circ\ell_\om (\te\cdot n)\\
\end{pmatrix} =
\begin{pmatrix}
\int_{\pa\om} g_1 \ \te\cdot n\\
\vdots \\
\int_{\pa\om} g_m \ \te\cdot n\\
\end{pmatrix}
= \int_{\pa\om} g_\om\ \te\cdot n,
\end{equation}
where $g_\om:\pa\om\to\RR^m$ denotes the vector valued function such that $\pi_k\circ g=g_k$ for all $k=1,2,\dots,m$. Finally, since $C^1(\ol U,\rn)$ is dense in $W_0^{1,\infty}(U,\rn)\supset \Theta$, we get that identity \eqref{J' on C^1} holds for all $\te\in\Theta$ as well, which is what we wanted to show.
\end{remark}

\begin{remark}
As shown in \cite{structure theorem new} by using tools from geometric measure theory, an analogous result holds under the weaker assumption that $\om$ is just a set of finite perimeter.
\end{remark}

We will conclude this subsection by giving a corollary of Theorem \ref{struct thm}, characterizing the structure of the Fr\'echet derivative of a shape functional evaluated at some point other than zero.
\begin{corollary}\label{structure corollary}
Let the notation be as in Theorem \ref{struct thm}. Suppose that $\cJ$ is Fr\'echet differentiable at some small $\varphi\in\Theta'$ and set $\cJ_\varphi(\theta):=\cJ(\varphi+\theta)$. Then, $\cJ_\varphi$ is also Fr\'echet differentiable at $0\in C^1(\ol U,\rn)$ and there exists a continuous linear map $\ell_\varphi:C^1(\pa\om)\to \RR^m$ such that
\begin{equation*}
\cJ'(\varphi)[\te]=\cJ_\varphi'(0)[\te]=\ell_{\om_\varphi}\left( \restr{\te \circ(\id+\varphi)^{-1}}{\pa\om_\varphi} \cdot n_\varphi \right)\quad \text{for all }\theta\in C^1(\ol U,\rn),
\end{equation*}
where $n_\varphi$ denotes the outward unit normal vector to the perturbed set $\pa\om_\varphi$.
\end{corollary}
\begin{proof}
Fix $\varphi\in\Theta'$ and $\theta\in C^1(\ol U,\rn)$. By construction, it is clear that $\cJ_\varphi$ is Fr\'echet differentiable at $0\in C^1(\ol U,\rn)$ and $\cJ'(\varphi)[\te]=\cJ_\varphi'(0)[\te]$.
In what follows, we will compute the Fr\'echet derivative $\cJ_\varphi'(0)[\te]$ as a G\^ateaux derivative in order to simplify the computations. We have
\begin{equation*}
\begin{aligned}
\cJ_\varphi'(0)[\te]&= \dato \cJ_\varphi(t\te)=\dato \cJ(\varphi+t\te) = \dato J\pp{(\id+ \varphi+t\te)\ \om} \\
&= \dato J\pp{(\id+\varphi+t\te)\circ(\id+\varphi)\inv \om_\varphi}
=\dato J\pp{\pp{ \id+t\te\circ(\id+\varphi)\inv} \om_\varphi} \\
&= \ell_{\om_\varphi}\pp{\te\circ(\id + \varphi)\inv \cdot n_\varphi},
\end{aligned}
\end{equation*}
where in the fourth equality we employed the fact that the map $\id+\varphi$ is a bijection from $\om$ to $\om_\varphi$ (for $\varphi$ small enough) and in the last equality we made use of Theorem \ref{struct thm} applied to the set $\om_\varphi$.
\end{proof}

\subsection{Proof}
The proof of Theorem \ref{thm I} follows by an application of the following version of the implicit function theorem for Banach spaces (\cite[Theorem 2.3]{AP1983}).

\begin{thm}[Implicit function theorem]\label{ift}
Let $\Psi\in C^k(\La\times W,\cH)$, $k\ge 1$, where $\cH$ is a Banach space and $\La$ (resp. $W$) is an open set of a Banach space $\cF$ (resp. $\cG$). Suppose that $\Psi(f^*,g^*)=0$ and that the partial derivative $\pa_g\Psi(f^*,g^*)$ is a bounded invertible linear transformation from $\cG$ to $\cH$.
Then there exist neighborhoods $\Theta$ of $f^*$ in $\cF$ and $W^*$ of $g^*$ in $\cG$, and a map $\widetilde g\in\cC^k(\Theta,\cG)$ such that the following hold:
\begin{enumerate}[(1)]
\item $\Psi(f,\widetilde g(f))=0$ for all $f\in\Theta$,
\item If $\Psi(f,g)=0$ for some $(f,g)\in\Theta\times W^*$, then $g=\widetilde g(f)$,
\item $(\widetilde g)'(f)=-[\pa_g \Psi(p) ]^{-1}\circ \pa_f \Psi(p)$, where $p=(f,\widetilde g(f))$ and $f\in\Theta$.
\end{enumerate}
\end{thm}
\begin{remark}
Theorem \ref{ift} also holds when $k=\infty$ even though this is not explicitly stated in the statement of \cite[Theorem 2.3]{AP1983}. The proof is simple. Indeed, since the neighborhoods $\Theta$, $W^*$ and the map $\widetilde g:\Theta\to\cG$ in the theorem do not depend on $k$ (see the proof of \cite[Theorem 2.3 and Lemma 2.1]{AP1983} for the details), if $\Psi$ is of class $C^\infty$ then the map $\widetilde g$ belongs to $C^k(\Theta,\cG)$ for all $k\ge1$. In other words, $\widetilde g \in C^\infty(\Theta,\cG)$.
\end{remark}
\begin{proof}[Proof of Theorem \ref{thm I}]
By hypothesis, we have that $\cJ_\Om:\Theta'\times \La'\to\RR^m$ is Fr\'echet differentiable in a neighborhood of $(0,0)\in \Theta\times\La$. By composition, this implies that also the mapping \begin{equation*}
j(\xi,\la):=\cJ_\Om\pp{E(\xi,\la), \la}
\end{equation*}
is Fr\'echet differentiable in a neighborhood of $(0,0)\in X'\times \La'$. Computing the first partial derivative with respect to the first variable yields
\begin{equation*}
\begin{aligned}
\pa_x j (\xi,\la)[\eta]= \pa_\te \cJ_\Om\pp{E(\xi,\la),\la}\qq{\pa_x E(\xi,\la)[\eta]}\\
= \int_{\pa\Om_{\xi,\la}} \pp{ g_{\Om_{\xi,\la}}(\la) \ \pp{\pa_x E(\xi,\la)[\eta]}\circ \restr{\pp{\id+E(\xi,\la)\inv}}{\pa\Om_{\xi,\la}} \cdot n_{\xi,\la}},
\end{aligned}
\end{equation*}
where $n_{\xi,\la}$ denotes the outward unit normal vector to $\pa\Om_{\xi,\la}$.
By a change of variables, the expression above can be rewritten as
\begin{equation}\label{pax j}
\pa_x j(\xi,\la)[\eta]= \int_{\pa\Om} h(\xi,\la)\ m(\xi,\la)\cdot \restr{\pp{\pa_x E(\xi,\la)[\eta]}}{\pa\Om}\ ,
\end{equation}
where
\begin{equation}\label{h&m}
h(\xi,\la):= g_{\Om_{\xi,\la}}(\la) \circ \pp{\id+E(\xi,\la)}, \quad m(\xi,\la):= J_\tau(\xi,\la)\ n_{\xi,\la}\circ\pp{\id+E(\xi,\la)}.
\end{equation}
Here $J_\tau(\xi,\la)$ denotes the tangential Jacobian associated to the map $\id+ E(\xi,\la)$ (see \cite[Definition 5.4.2 and Proposition 5.4.3]{HP2018}).
It is known (see \cite[Proposition 5.4.14 and Lemma 5.4.15]{HP2018}) that both the normal vector and the tangential Jacobian are Fr\'echet differentiable with respect to perturbations of class $C^1$. Moreover, by hypothesis, we know that also the mapping
\begin{equation*}
X'\times \La'\ni (\xi,\la)\mapsto \underbrace{g_{\Om_{\xi,\la}}(\la)\circ\pp{\id+E(\xi,\la)}}_{:=h(\xi,\la)}\in Y
\end{equation*}
is Fr\'echet differentiable.
By composition, both $h(\cdottone,\cdottone)$ and $m(\cdottone,\cdottone)$ are Fr\'echet differentiable in a neighborhood of $(0,0)\in X\times \La$. In particular, this implies that, for fixed $\eta$, also $\pa_x j(\cdottone,\cdottone)[\eta]$ is Fr\'echet differentiable in a neighborhood of $(0,0)\in X\times \La$.
Now, since, by hypothesis, $E(\cdottone,0)$ is a bounded linear operator satisfying \eqref{restr E}, we have
\begin{equation*}
\restr{\pp{\pa_x E(\xi,0)[\eta]}}{\pa\Om} = \restr{E(\eta,0)}{\pa\Om} = \eta n.
\end{equation*}
Thus, evaluating \eqref{pax j} at $(\xi,0)$ yields
\begin{equation*}
\pa_x j(\xi,0)[\eta]= \int_{\pa\Om} h(\xi,0)\ m(\xi,0)\cdot n\eta.
\end{equation*}
Differentiating the expression above with respect to the first variable one more time at the point $(0,0)$ yields
\begin{equation*}
\pa_{xx}^2\ j(0,0)[\xi,\eta]= \int_{\pa\Om} \pp{\pa_x h (0,0)[\xi]\ m(0,0)+h(0,0)\ \pa_x m(0,0)[\xi]}\cdot n \eta = \int_{\pa\Om} \underbrace{\pa_x h(0,0)[\xi]}_{:=Q(\xi)}\ \eta,
\end{equation*}
where we have made use of the following identities:
\begin{equation*}
h(0,0)=g_\Om(0)=0,\quad m(0,0)= n.
\end{equation*}
In other words, the bounded linear operator $Q$ defined in \eqref{Q} is nothing but $\pa_x h(0,0)$.
Now, by the nondegeneracy hypothesis $(ii)$, we can apply Theorem \ref{ift} to the mapping $h:X'\times \La'\to Y$. This yields the existence of neighborhoods $X''$ of $0\in X'$ and $\La''$ of $0\in\La'$ and of a mapping $\widetilde \xi: \La''\to X''$ such that
\begin{equation}\label{almost the conclusion}
\setbld{(\xi,\la)\in X''\times \La''}{ h(\xi,\la)=0} = \setbld{\pp{\widetilde\xi(\la),\la}}{\la\in\La''}.
\end{equation}
Now, recall that, by the first identity in \eqref{h&m} we have:
\begin{equation*}
h(\xi,\la)=0 \quad \ton \pa\Om \iff g_{\Om_{\xi,\la}}(\la) \quad \ton \pa\Om_{\xi,\la}.
\end{equation*}
In other words, \eqref{almost the conclusion} can be rewritten as
\begin{equation*}
\setbld{(\xi,\la)\in X''\times \La''}{ g_{\Om_{\xi,\la}}(\la) \quad \ton \pa\Om_{\xi,\la}} = \setbld{\pp{\widetilde\xi(\la),\la}}{\la\in\La''},
\end{equation*}
which concludes the proof of Theorem \ref{thm I}.
\end{proof}

\section{Application to Problem \ref{pb 2} in the nondegenerate case}
\subsection{Preliminaries: Problem \ref{pb 2} is variational}
Let $(D,\om)$ be a pair of bounded open sets satisfying $\ol D\subset\om$. Let $\pa\om$ be at least of class $C^{2,\al}$ and let $U$ be a bounded open neighborhood of $\om$ with smooth boundary that does not intersect $\ol D$. For fixed $\sg_c\in\RR$ and $f\in C^{1,\al}(\ol U)$ define the following shape functional:
\begin{equation}\label{two phase torsional rigidity}
J(\om, D,f,s):= \int_{\om} \sg |\gr u|^2-\int_\om (c+f)^2, \quad \sg:=(\sg_c+s)\cX_D+\cX_{\rn\setminus D},
\end{equation}
where $u$ is the solution to the boundary value problem \eqref{2ph eq}.
We claim that the pair $(D,\om)$ is a solution to Problem \ref{pb 2} with respect to parameters $(D,f,s)$ if and only if $\om$ is a critical point of the parametrized shape functional \eqref{two phase torsional rigidity} in the sense of \eqref{first derivative}.

Now, for small $\te\in \Theta:=W_0^{1,\infty}(U,\rn)$ and $\la=(\varphi,f,s)\in\La$, let $u_{\te,\la}$ denote the solution to the boundary value problem \eqref{2ph eq}--\eqref{sigma pert} with respect to the pair $(D,\om)$. Moreover, consider the function
\begin{equation}\label{kore}
v_{\te,\la}:= u_{\te,\la}\circ(\id+\varphi+\te).
\end{equation}
It is easy to show (see \cite[Theorem 5.3.2, the subsequent remark and remark 5.3.6]{HP2018}) that the map
\begin{equation*}
(\te,\la)\mapsto v_{\te,\la}\in H_0^1(\om)
\end{equation*}
is Fr\'echet differentiable infinitely many times in a small neighborhood of $(0,0)\in \Theta\times\La$.

In the literature, the derivative of $v_{\te,\la}$ is usually referred to as the \emph{material derivative} of $u_{\te,\la}$, while the function
\begin{equation}\label{u'}
u_\la'[\te_0]:=\restr{\frac{\pa}{\pa\te}}{\te=0} v_{\te,\la}[\te_0]-\gr u_{0,s}\cdot \te_0 \quad \text{for all }\te_0\in W^{1,\infty}(\rn,\rn),
\end{equation}
is referred to as the \emph{shape derivative} of $u_{\te,\la}$.
It is known that $u_\la'[\te_0]\in L^2(\om)$ and that, for all bounded open sets $A\subset \ol A \subset\om\setminus \pa D$, the map $\te\mapsto \restr{u_{\te,\la}}{A}\in H^1(A)$ is Fr\'echet differentiable at $0\in \Theta$ and its derivative coincides with the restriction to $A$ of \eqref{u'}.
It is known that the function $u_\la'[\te_0]$ can be characterized as the unique solution to the following boundary value problem:
\begin{equation}\label{u'}
\begin{cases}
-\dv\left( \sg\gr u'\right)=0 \quad \text{in }\om, \\
u'=-\gr u \cdot \te_0 \quad \text{on }\pa\om.\\
\end{cases}
\end{equation}
A result analogous to that of Theorem \ref{struct thm} holds for the function $u_\la'$ as well. Namely, we know that $u_\la'[\te_0]$ depends on $\te_0$ only by means of its normal component $\te_0\cdot n_\om$ on $\pa\om$. For this reason, if $\restr{\te_0}{\pa\om}=\xi n_\om$ for some function $\xi$, then, without ambiguity, we will use the notation $u_\la'[\xi]$ to refer to $u_\la'[\te_0]$ instead.

Let $\cJ_{\om}(\te,\la):=J(D_\varphi,\om_\te,f,s)$. By the above discussion, we infer that the functional
\begin{equation*}
\cJ_{\om}(\te,\la)= (\sg_c+s)\int_{D_\varphi} |\gr u_{\te,\la}|^2 + \int_{\om\setminus \ol{D_\varphi}} |\gr u_{\te,\la}|^2 - \int_{\om_\te} (c+f)^2
\end{equation*}
is Fr\'echet differentiable (actually, infinitely many times) in a neighborhood of $(0,0)\in \Theta\times \La$.

Finally, the well-known Hadamard's formula \cite[Theorem 5.2.2]{HP2018} allows us to obtain the expression for the Fr\'echet partial derivative with respect to the first variable. We get
\begin{equation*}
\pa_\te \cJ_{\om}(0,\la)[\te_0]=
2 \underbrace{\int_\om \sg\gr u_\la \cdot \gr u_\la'[\te_0]}_{=0} +
\int_{\pa\om} |\gr u_\la|^2\ \te_0\cdot n_\om -\int_{\pa\om} (c+f)^2\ \te_0\cdot n_\om
\end{equation*}
Notice that the first integral in the expression above vanishes. This can be shown by taking $u_\la\in H_0^1(\om)$ as a test function in the weak form of \eqref{u'}.
As a result, we can write
\begin{equation*}
\pa_\te \cJ_{\om}(0,\la)[\te_0]=
\int_{\pa\om} \underbrace{\left( |\gr u_\la|^2 - (c+f)^2 \right)}_{:=g_{\om}(\la)}\ \te_0\cdot n_\om.
\end{equation*}
By the arbitrariness of $\te_0$, we conclude that Problem \ref{pb 2} is a variational problem associated with the parametrized shape functional $J$.

\subsection{Defining nondegeneracy for Problem \ref{pb 2}}
Let $(D,\Om)$ be a solution to Problem \ref{pb 2}. For simplicity we will assume that $\pa\Om$ is at least of class $C^{3,\al}$ (we remark that this assumption is not restrictive, since, in light of \cite[Theorem 2]{KN77}, the boundary of every classical solution of Problem \ref{pb 2} must indeed be analytic).

In what follows, we will study the Fr\'echet derivative of the map
\begin{equation}\label{mapsto}
(\te,\la)\mapsto g_{\Om_\te}(\la)\circ\pp{\id+\varphi+\te}= \pp{|\gr u_{\te,\la}|^2-(c+f)^2}\circ \pp{\id+\varphi+\te} \in C^{1,\al}(\pa\Om)
\end{equation}
in the appropriate function spaces.
To this end, we will make use of the following lemma.
Let $\Theta_{\rm reg}$ be defined as in \eqref{some definitions} and let $v_{\te,\la}$ denote the function defined by \eqref{kore}. We have the following result.
\begin{lemma}\label{v is shape smooth}
The map
\begin{equation*}
(\te,\la)\mapsto v_{\te,\la}\in H_0^1(\Om)\cap C^{2,\al}(\ol \Om\cap \ol U)
\end{equation*}
is Fr\'echet differentiable infinitely many times in a neighborhood of $(0,0)\in\Theta_{\rm reg}\times \La$.
\end{lemma}
\begin{proof}[Sketch of the proof]
This can be proved in a standard way using Theorem \ref{ift} and the core idea is not different from the proof of the Fr\'echet differentiability of the map $(\te,\la)\mapsto v_{\te,\la}\in H_0^1(\Om)$ of the previous subsection. The only difference lies in using the right Schauder estimates to show that the restriction of $v_{\te,\la}$ to $\ol \Om\cap\ol U$ varies smoothly in the $C^{2,\al}$ norm as well.
We refer to \cite[Appendix]{SAP} for the details.
\end{proof}
By the lemma above, the Fr\'echet differentiability of the map \eqref{mapsto} ensues once we are able to rewrite $\gr u_{\te,\la}\circ(\id+\varphi+\te)$ in terms of $v_{\te,\la}$. By the chain rule, we have
\begin{equation}\label{gr u hikimodoshi}
\gr u_{\te,\la}\circ(\id+\varphi+\te)= (I+D\varphi+D\te)^{-T} \gr v_{\te,\la},
\end{equation}
where $I\in\RR^{N\times N}$ is the identity matrix and the superscript $-T$ denotes the transposed inverse matrix.
Finally, combining \eqref{gr u hikimodoshi} and Lemma \ref{v is shape smooth} yields the Fr\'echet differentiability of the map $\Theta_{\rm reg}\times \La\to X$ defined by \eqref{mapsto}.

In what follows, we compute the partial Fr\'echet derivative of the map \eqref{mapsto} with respect to the first variable at $(0,0)\in \Theta_{\rm reg}\times \La$.
To this end, notice that, by \eqref{gr u hikimodoshi}, we have
\begin{equation*}
g_{\Om_\te}(0)\circ(\id+\te)= |(I+D\te)^{-T}\gr v_\te|^2-c^2 \in X.
\end{equation*}
Differentiating with respect to $\te$ at $\te=0$ yields
\begin{equation*}
\restr{\frac{\pa}{\pa\te}}{\te=0}\bigg( g_{\Om_\te}(0)\circ(\id+\te)\bigg) [\te_0]= 2\gr u \cdot \pp{ (-D\te_0)^T\gr u + \gr v'[\te_0]} = 2\gr u \cdot \pp{\gr u'[\te_0]+ D^2 u\ \te_0 },
\end{equation*}
where, in the second equality, we made use of the identity
\begin{equation*}
\gr u'[\te_0]=\gr v'[\te_0]-(D\te_0)^T \gr u - D^2 u\ \te_0,
\end{equation*}
which in turn is derived by taking the gradient of \eqref{u'}.

We are now ready to compute the partial derivative of the map $h$, which is defined as in \eqref{h&m} for some bounded linear extension operator $E:X\times \La \to \Theta_{\rm reg}$ that satisfies $\restr{E(\xi,0)}{\pa\Om}=\xi n$. The calculations above yield
\begin{equation*}
Q(\xi):=\pa_xh(0,0)[h]= 2\pa_n u \pp{\pa_n u'[\xi]+\pa_{nn}^2 u\ \xi }= -2c \pp{\pa_n u'[\xi]+\pa_{nn}^2 u\ \xi }.
\end{equation*}
In light of the above, \emph{nondegeneracy} for Problem \ref{pb 2} will be defined as follows:
\begin{definition}\label{def non deg}
Let $(D,\Om)$ be a solution to Problem \ref{pb 2}. We say that $(D,\Om)$ is nondegenerate if the map
\begin{equation}\label{Gamma}
\begin{aligned}
\Ga:\quad & C^{2,\al}(\pa\Om)\longrightarrow C^{1,\al}(\pa\Om) \\
& \xi\longmapsto \pa_n u'[\xi]+\pa_{nn}^2 u\ \xi,
\end{aligned}
\end{equation}
satisfies $\ker \Ga=\{0\}$.
\end{definition}

\subsection{Proof of Theorem \ref{thm II}}
In what follows, we will show that condition $\ker\Ga=\{0\}$ is indeed equivalent to the bijectivity of $Q$ and this will conclude the proof of Theorem \ref{thm II}.

Let $\iota:\C^{2,\al}(\pa\Om)\xhookrightarrow{\quad}\C^{1,\al}(\pa\Om)$ denote the inclusion mapping.
We have the following result.
\begin{lemma}
Let $\mu>\norm{\pa_{nn}^2u}_{L^\infty(\pa\Om)}$. Then the map
\begin{equation*}
\Ga+\mu\iota: C^{2,\al}(\pa\Om)\to C^{1,\al}(\pa\Om)
\end{equation*}
is a bijection.
\end{lemma}
\begin{proof}
For any given $\eta\in C^{1,\al}(\pa\Om)$, we will show that there exists a unique element $\xi\in C^{2,\al}(\pa\Om)$ that satisfies
\begin{equation*}
(\Ga+\mu\iota)\xi=\eta.
\end{equation*}
First of all, let us consider the Sobolev space $H^1(\Om)$ endowed with the (equivalent) norm $\norm{\varphi}_{H^1(\Om)}:=\norm{\gr\varphi}_{L^2(\Om)}+\norm{\restr{\varphi}{\pa\Om}}_{L^2(\pa\Om)}$ and define the following bilinear form:
\begin{equation*}
\begin{aligned}
\beta: \quad & H^1(\Om)\times H^1(\Om)\longrightarrow\RR\\
& (w,\varphi)\longmapsto \int_\Om \sg \gr w\cdot \gr \varphi + \int_{\pa\Om} -\frac{1}{\pa_n u} (\mu+\pa_{nn}^2 u)\ w\varphi.
\end{aligned}
\end{equation*}
Notice that, by the Hopf lemma, $\pa_n u$ never vanishes on $\pa\Om$ and thus $\be$ is well defined. Moreover, $\be$ is bilinear, continuous and coercive.
Fix now an element $\eta\in C^{1,\al}(\pa\Om)\subset L^2(\pa\Om)$. Then, by the Lax-Milgram theorem, there exists a unique $w\in H^1(\Om)$ such that
\begin{equation*}
\beta(w,\varphi)= \left\langle\eta,\restr{\varphi}{\pa\Om}\right\rangle_{L^2(\pa\Om)}\quad \text{for all }\varphi\in H^1(\Om).
\end{equation*}
Now, if we restrict the identity above to $\varphi$ in $H^1_0(\Om)$, then we realize that $w$ must satisfy
\begin{equation}\label{w eq}
-\dv(\sg\gr w)=0\quad \tin \Om.
\end{equation}
Moreover, integration by parts and the arbitrariness of the trace of $\varphi\in H^1_0(\Om)$ on $\pa\Om$ yield
\begin{equation}\label{w bc}
\pa_n w -\frac{1}{\pa_n u}(\mu+\pa_{nn}^2 u) w = \eta \quad \ton \pa\Om.
\end{equation}
Now, since $w$ is the solution to the boundary problem \eqref{w eq} and \eqref{w bc}, we can inductively bootstrap its regularity in a classical way by means of the standard elliptic regularity estimates and the Schauder boundary estimates (see for example the argument in the proof of \cite[Proposition 5.2]{kamburov sciaraffia} after $(5.7)$). We obtain that $w\in C^{2,\al}(\ol\Om \cap \ol U)$ for all open neighborhoods $U\supset \pa\Om$ whose boundary $\pa U$ is of class $C^{2,\al}$ and whose closure $\ol U$ does not intersect $\ol D$. In particular, the function \begin{equation*}
\xi:= \frac{\restr{w}{\pa\Om}}{-\pa_n u}
\end{equation*}
is a well defined element of $C^{2,\al}(\pa\Om)$. This, together with \eqref{w eq} and \eqref{w bc}, implies that $u'[\xi]=w$.
In particular, again by \eqref{w bc},
\begin{equation*}
(\Ga+\mu\iota) \xi = \pa_n w + \frac{\pa_{nn}^2 u + \mu}{-\pa_n u} \restr{w}{\pa\Om} = \eta.
\end{equation*}
By the arbitrariness of $\eta\in C^{1,\al}(\pa\Om)$, the above shows that $\Ga+\mu\iota: \C^{2,\al}(\pa\Om)\to\C^{1,\al}(\pa\Om)$ is a bijection.
\end{proof}

\begin{proposition}\label{characterization}
The map $\Ga:\C^{2,\al}(\pa\Om)\to\C^{1,\al}(\pa\Om)$ is a continuous bijection if and only if $\ker \Ga =\{0\}$.
\end{proposition}
\begin{proof}
The continuity of $\Ga$ holds by construction. On the other hand, characterizing the invertibility of $\Ga$ is not obvious. We want to show that, under the assumption $\ker\Ga=\{0\}$, for any given $\eta\in C^{1,\al}(\pa\Om)$ there exists a unique element $\xi\in C^{2,\al}(\pa\Om)$ such that $\Ga\xi=\eta$. To this end, take an element $\eta\in C^{1,\al}(\pa\Om)$ and let $\mu>\norm{\pa_{nn}^2 u}_{L^\infty(\pa\Om)}$ so that the map $\Ga+\mu\iota: C^{2,\al}(\pa\Om)\to C^{1,\al}(\pa\Om)$ is a bijection. Set $K:=\iota (\Ga+\mu\iota)\inv: C^{1,\al}(\pa\Om)\to C^{1,\al}(\pa\Om)$. Now, since $\pa\Om$ is compact, the inclusion mapping $\iota: C^{2,\al}(\pa\Om)\xhookrightarrow{\quad} C^{1,\al}(\pa\Om)$ is a compact operator and thus so is $K$. As a result, $\id-\mu K$ is a Fredholm operator of index 0 from $C^{1,\al}(\pa\Om)$ to itself.

In what follows, we will show that the operator $\id-\mu K$ is invertible. By the Fredholm alternative theorem, it will be sufficient to show that $\ker (\id-\mu K)=\{0\}$. To this end, take an element $\zeta \in \ker (\id-\mu K)$. We have
\begin{equation*}
(\id-\mu K)\zeta =0 \implies \iota\zeta=\zeta= \mu K \zeta = \mu \iota (\Ga+\mu \iota)\inv \zeta\implies (\Ga+\mu \iota) \zeta=\mu\iota\zeta \implies \Ga\zeta=0.
\end{equation*}
Now, since, by assumption, $\ker \Ga=\{0\}$, we conclude that $\zeta=0$, that is, also $\ker (\id-\mu K)=\{0\}$. By the Fredholm alternative theorem, this implies that $\id-\mu K: C^{1,\al}(\pa\Om)\to C^{1,\al}(\pa\Om)$ is a bijection.

Let now $\eta$ be a fixed element of $C^{1,\al}(\pa\Om)$ and set
\begin{equation}\label{tilde xi}
{\xi}^\star:= (\id-\mu K)\inv K \eta \in C^{1,\al}(\pa\Om).
\end{equation}
We claim that the function ${\xi}^\star$ actually belongs to $C^{2,\al}(\pa\Om)$. Indeed, we can rearrange the terms in \eqref{tilde xi} to get
\begin{equation}\label{tilde xi =}
{\xi}^\star= K(\eta+\mu {\xi}^\star) = \iota \underbrace{ (\Ga+\mu\iota)\inv (\eta+\mu{\xi}^\star)}_{\in C^{2,\al}(\pa\Om)}.
\end{equation}
In other words, ${\xi}^\star$ lies in the image $\iota\pp{C^{2,\al}(\pa\Om)}$ (which is just the space $C^{2,\al}(\pa\Om)$ seen as a subspace of $C^{1,\al}(\pa\Om)$). This identification between $\iota\pp{C^{2,\al}(\pa\Om)}$ and $C^{2,\al}(\pa\Om)$ allows us to write that, again, by \eqref{tilde xi =},
\begin{equation*}
(\Ga+\mu\iota){\xi}^\star = \eta +\mu {\xi}^\star \implies \Ga{\xi}^\star=\eta.
\end{equation*}
By the arbitrariness of $\eta\in C^{1,\al}(\pa\Om)$ we conclude that $\Ga$ is a bijection, as claimed.
\end{proof}

\section{Remarks on the degenerate case}
The focus of this section will be on the case when the given shape $\Om$ is a \emph{degenerate} critical point for some parametrized shape functional. In particular, we will show why this behavior might occur and how (and to what extent) we can circumvent this problem. Here we give a proof of Theorem \ref{thm III}.
\subsection{Degeneracy due to bifurcation phenomena}
If $\Om$ is a \emph{degenerate} critical point, then Theorem \ref{thm I} cannot be applied. Usually, the reason for this is not technical but is rather the reflection of the local behavior of the family of solutions near $\Om$. As a matter of fact, most of the times, \emph{the} parametrized branch of solutions starting from $\Om$ is not uniquely defined. This is what happens when $\Om$ is a \emph{bifurcation point}, that is, there exist multiple families of solutions bifurcating from $\Om$.

The bifurcation analysis for Problem \ref{pb 2} (when $\sg_c\ne 1$) around concentric balls has been performed in \cite{CYisaac}. Indeed, in light of Theorem \ref{thm II}, we can say that such bifurcation phenomena arise ``whenever possible", that is, whenever the concentric balls $(D_0,\Om_0)$ constitute a degenerate critical point of the torsional rigidity functional in the sense of Definition \ref{def non deg} (see also \cite{cava2018}).

On the other hand, the bifurcation analysis when $\sg_c=1$ is ``trivial". By \cite{Se1971}, we know that $\Om$ is a solution if and only if $\Om$ is a ball (whose radius is uniquely determined by the parameter $c$ of the problem). As a result, every solution is also a bifurcation point, from which infinitely many branches emanate. Finally, by identifying each ball with its center, we can conclude that the set of solutions of Problem \ref{pb 2} when $\sg_c=1$ forms a topological space that is homeomorphic to $\rn$. Under this interpretation, we can say that any curve starting from a point $x_0\in\rn$ uniquely identifies a branch of solutions that bifurcates from the ball of center $x_0$.

Now, roughly speaking, Theorem \ref{thm I} tells us that, any ``perturbed problem" inherits the degeneracy behavior of the ``original problem". Therefore, it is clear that, for Problem \ref{pb 2}, complications arise in a neighborhood of $\sg_c=1$. In the next subsection, we are going to show how this affects the behavior of the parametrized families of solutions in a neighborhood of $\sg_c=1$ and what compromises must be made in order to construct them.

\subsection{Restricting the perturbation space to circumvent the problem}

In this subsection, we will give a proof of Theorem \ref{thm III}. Let $\Om$ be a ball and $D$ be an open set with $\ol D\subset\Om$ and $\sg_c=1$. For simplicity, let us consider the case where the radius of $\Om$ is normalized to $1$.
By elementary computations we know that
the solution of \eqref{2ph eq} can be explicitly written as
\begin{equation*}
u(x)=\frac{1-|x|^2}{2N}, \quad \tfor x\in \ol\Om.
\end{equation*}
In other words, the pair $(D,\Om)$ is a solution to Problem \ref{pb 2} for $c=1/N$ (indeed, for any open set $D\subset \ol D\subset \Om$).

Let the sets $\La$, $X$ and $Y$ be defined following \eqref{notaton 1}--\eqref{notation 2}.
In this special setting, the map $Q$ (defined as in the previous section) can be written as
\begin{equation*}
Q(\xi)=-\frac{2}{N} \left( \pa_n u'[\xi] -\frac{1}{N} \xi\right).
\end{equation*}
Here $u'[\xi]$ is just the solution to \eqref{u'} in the special case where $\sg\equiv 1$, $\om$ is the unit ball and $\te_0=\xi n$. It is well known that this boundary value problem admits an explicit solution by means of a spherical harmonics expansion (see \cite[Proposition 3.2]{cava2018} where a more general case is considered). We have
\begin{equation*}
u'\left[ \sum_{k=0}^\infty \sum_{i=1}^{d_k} \al_{k,i}Y_{k,i} \right](x) = \frac{1}{N}\sum_{k=0}^\infty \sum_{i=1}^{d_k} \al_{k,i}\ |x|^k \ Y_{k,i}\pp{\frac{x}{|x|}} \quad \tfor x\in\Om\setminus\{0\}.
\end{equation*}
Here $Y_{k,i}$ are the so-called spherical harmonics, defined as the solutions to the following eigenvalue problem for the Laplace--Beltrami operator on the unit sphere
\begin{equation*}
-\De_\tau Y_{k,i}= \la_{k} Y_{k,i} \quad \ton \pa\Om
\end{equation*}
and normalized such that $\norm{Y_{k,i}}_{L^2(\pa\Om)}=1$. Furthermore, the eigenspace $\cY_k(\pa\Om)$ corresponding to the $k$\textsuperscript{th} eigenvalue has dimension $d_k$ and is spanned by $\langle Y_{k,1}, \dots, Y_{k,d_k}\rangle$.

By the above, $Q$ admits the following extension as a bounded linear map $H^1(\pa\Om)\to L^2(\pa\Om)$:
\begin{equation*}
Q\pp{ \sum_{k=0}^\infty \sum_{i=1}^{d_k} \al_{k,i}Y_{k,i} } = -\frac{2}{N^2} \sum_{k=0}^\infty \sum_{i=1}^{d_k} (k-1)\ \al_{k,i}Y_{k,i}.
\end{equation*}
As a result, its kernel is $\ker Q = \cY_1(\pa\Om)$, which is an $N$-dimensional space (bearing a one-to-one correspondence with the set of translations in $\rn$).

In other words, $\Om$ is a degenerate critical point and Theorem \ref{thm I} cannot be applied.
To circumvent this problem, instead of $h$, we consider the following modified function
\begin{equation*}
h_{\rm bar}(\xi,\la):= \pi_{\rm bar} \circ g_{\Om_{\xi,\la}}(\la) \circ \pp{\id+(\xi+\eta)n},
\end{equation*}
where $\pi_{\rm bar}:C^{1,\al}(\pa\Om)\to C_{\rm bar}^{1,\al}(\pa\Om)$ is the projection operator onto the space $C_{\rm bar}^{1,\al}(\pa\Om)$ defined in \eqref{notation 2} and $g_{\Om_{\xi,\la}}=|\gr u_{\xi,\la}|^2-(c+f)^2$.

Since $\pi_{\rm bar}$ is a bounded linear operator, we get
\begin{equation*}
Q_{\rm bar}:= \pa_x h_{\rm bar}(0,0) = \pi_{\rm bar} \circ \pa_x h(0,0)= \pi_{\rm bar}\circ Q.
\end{equation*}
By the above, $\ker Q_{\rm bar}=\{0\}$. Now, along the same lines as Proposition \ref{characterization} we get that $Q_{\rm bar}$ is a continuous bijection between $C_{\rm bar}^{2,\al}(\pa\Om)$ and $C_{\rm bar}^{1,\al}(\pa\Om)$. Finally, by reasoning as in the final stage of the proof of Theorem \ref{thm II}, we deduce that there exist neighborhoods $X''$ and $\La''$, and a mapping $\widetilde\xi: \La''\to X''$ such that \eqref{almost the conclusion} holds for $h_{\rm bar}$. This concludes the proof of Theorem \ref{thm III}.
\begin{remark}
The exponent $2$ in \eqref{odc bar} does not have any particular meaning. Indeed, Theorem \ref{thm III} would still be true if both occurrences of the exponent $2$ in \eqref{odc bar} were replaced by any number $p\in(0,\infty)$. More generally, Theorem \ref{thm III} would still be true (and the proof would follow verbatim from the current one) if we replaced the entire \eqref{odc bar} with
\begin{equation*}
\big( F(|\gr u_{\xi,\la}|)- F(c+f)\big)\circ \pp{\id+E(\xi,\la)} \in \cY_1(\pa\Om),
\end{equation*}
where $F:(0,\infty)\to \RR$ is a differentiable function satisfying $F'(c)\ne 0$.
We remark that, in this case, the resulting $\widetilde\xi$ would depend on $F$.
\end{remark}
\begin{remark}
An analogous result to Theorem \ref{thm III} in the one-phase case has recently been proven by Gilsbach and Onodera in \cite{gilsbach onodera}, where they also provide quantitative stability estimates for the solution.
\end{remark}
\section{Appendix: bounded linear extension operators}
In this section, we will show how to construct a bounded linear extension operator like the one that was used in sections 3 and 4. We will use the following result (\cite[Lemma 6.38]{GT}):
\begin{lem}\label{extension lemma}
Let $\Om\subset\rn$ be a bounded open set of class $C^{k,\al}$ ($k\ge 1$) and let $U$ be an open set containing $\ol\Om$. Suppose $\varphi\in C^{k,\al}(\pa\Om)$. Then there exists a bounded linear operator
\begin{equation*}
E_{\pa\Om}: C^{k,\al}(\pa\Om)\to C_0^{k,\al}(U)
\end{equation*}
such that $\restr{E_{\pa\Om}(\varphi)}{\pa\Om}=\varphi$.
\end{lem}
\begin{remark}
Originally, \cite[Lemma 6.38]{GT} gives the result only in the case where $\Om$ is a \emph{domain} of class $C^{k,\al}$. Actually, the assumption of connectedness is not needed if we use the following definition of open set of class $C^{k,\al}$. We say that an open set $\Om\subset \rn$ is of class $C^{k,\al}$ (for all $k\ge0$ and $0<\al<1$) if at each point $x_0\in\pa\Om$ there is a ball $B=B(x_0,r)$ and a one-to-one mapping $\psi$ of $B$ onto $\om\subset\rn$ such that the following hold:
\begin{equation}\label{regularity}
\begin{aligned}
(i)\quad \psi(B\cap \Om)\subset
\{x_N>0\}, \quad (ii)\quad \psi(B\cap \pa\Om)\subset
\{x_N=0\},\\
(iii) \quad \psi\in C^{k,\al}(B),\quad (iv)\quad \psi^{-1}\in C^{k,\al}(\om).
\end{aligned}
\end{equation}
\end{remark}
\begin{remark}
Actually, \cite[Lemma 6.38]{GT} does not explicitly state neither boundedness nor linearity for the extension operator. Still, those properties are a straightforward consequence of the way the extension is constructed (see also \cite[Lemma 6.37]{GT}). Furthermore, we remark that the operator norm of $E_{\pa\Om}$ can be estimated by means of the radius $r$ and the norms $\norm{\psi}_{C^{k,\al}}$, $\norm{\psi^{-1}}_{C^{k,\al}}$ in \eqref{regularity} corresponding to those balls $B(x_0,r)\subset U$ that satisfy \eqref{regularity} (in particular, the operator norm of $E_{\pa\Om}$ also depends on the set $U$ by construction).
\end{remark}

Let now $\Om$ be a $C^{3,\al}$ domain of $\rn$ and let $U$ be an open set containing $\ol\Om$ (notice that the set $U$ can be taken to be arbitrarily ``small").
The bounded linear extension operator $E:C^{2,\al}(\pa\Om)\to \Theta_{\rm reg}$ will be then constructed as follows.
Since $\pa\Om$ is of class $C^{3,\al}$, its outward unit normal is a well defined function in $C^{2,\al}(\pa\Om,\rn)$.
Then, for $\xi\in C^{2,\al}(\pa\Om)$, we set
\begin{equation*}
E(\xi):=\begin{pmatrix}
E_{\pa\Om}(\xi n_1)\\
\vdots\\
E_{\pa\Om}(\xi n_N)
\end{pmatrix}.
\end{equation*}
By construction, $E(\cdottone)$ is bounded, linear and satisfies $\restr{E(\xi)}{\pa\Om}=\xi n$ for all $\xi\in C^{2,\al}(\pa\Om)$.

\section*{Acknowledgements}
The author is partially supported by JSPS Grant-in-Aid for Research Activity Startup Grant Number JP20K22298.

\begin{small}

\end{small}

\bigskip

\noindent
\textsc{
Mathematical Institute, Tohoku University, Aoba,
Sendai 980-8578, Japan } \\
\noindent
{\em Electronic mail address:}
cavallina.lorenzo.e6@tohoku.ac.jp

\end{document}